\documentclass{amsart}
\date{\today}
\title[Characteristic forms]%
{\texorpdfstring%
{Characteristic forms of complex Cartan geometries III: $G$-structures}%
{Characteristic forms of complex Cartan geometries III: G-structures}}
\newcommand{\authorsname}{\texorpdfstring{Benjamin \scotsMc{}Kay}{Benjamin McKay}}
\author{\authorsname}
\address{University College Cork}
\email{b.mckay@ucc.ie}
\thanks{Thanks to Francesco Russo for his invitation to carry out this work at the University of Catania.}
\usepackage{lmodern}
\usepackage[T2A,T1]{fontenc}
\usepackage[utf8]{inputenc}
\usepackage[tt={variable=false}]{cfr-lm}
\usepackage{xparse}
\usepackage{mathtext}
\usepackage{mathrsfs}
\usepackage[english]{babel}
\usepackage{varioref}
\usepackage[pdftex,pagebackref]{hyperref}
\hypersetup{
colorlinks   = true,  
urlcolor     = black, 
linkcolor    = black, 
citecolor    = black  
}
\usepackage[kerning=true,tracking=true]{microtype}
\usepackage{xspace}
\usepackage{amsfonts}
\usepackage{verbatim}
\usepackage{amsthm}
\usepackage{amssymb}
\usepackage{mathtools}
\usepackage{mathabx}
\usepackage{braket}
\usepackage{cool}
\usepackage{xstring}
\usepackage{array}
\usepackage{ragged2e}
\usepackage{longtable}
\usepackage{multirow}
\usepackage{booktabs}
\usepackage{bm}
\usepackage{tikz-cd}
\usepackage{varwidth}
\usepackage{tensor}
\usepackage[super]{nth}
\usepackage{scalerel}
\usepackage{rank-2-roots}
\usepackage{dynkin-diagrams}
\usepackage{cleveref}

\numberwithin{equation}{section}

\makeatletter
\DeclareRobustCommand{\scotsMc}{\scotsMcx{c}}
\DeclareRobustCommand{\scotsMC}{\scotsMcx{\textsc{c}}}
\DeclareRobustCommand{\scotsMcx}[1]{%
  M%
  \raisebox{\dimexpr\fontcharht\font`M-\height}{%
    \check@mathfonts\fontsize{\sf@size}{0}\selectfont
    \kern.3ex\underline{\kern-.3ex #1\kern-.3ex}\kern.3ex
  }%
}
\expandafter\def\expandafter\@uclclist\expandafter{%
  \@uclclist\scotsMc\scotsMC
}
\makeatother

\newtheorem{theorem}{Theorem}
\newtheorem{corollary}{Corollary}
\newtheorem{lemma}{Lemma}

\theoremstyle{remark}

\newtheorem{example}{Example}

\newcounter{remarkCounter}
\NewDocumentCommand\pr{m}{\ensuremath{\left(#1\right)}}
\NewDocumentCommand\of{m}{\ensuremath{\!\pr{#1}}}
\RenewDocumentCommand\C{o}%
{%
	\ensuremath%
	{\IfValueTF{#1}{\mathbb{C}^{#1}}{\mathbb{C}}}%
}%
\NewDocumentCommand\OO{D(){0}O{1}}%
{
\ensuremath{
\mathcal{O}%
\ifnum\pdfstrcmp{#1}{0}=0{}\else{(#1)}\fi
\ifnum\pdfstrcmp{#2}{1}=0{}\else{
	\ifnum\pdfstrcmp{#2}{0}=0{}\else{^{\oplus #2}}\fi
}\fi
}
}
\ifdefined\Proj
\RenewDocumentCommand\Proj{m}{\ensuremath{\mathbb{P}^{#1}}}
\else
\NewDocumentCommand\Proj{m}{\ensuremath{\mathbb{P}^{#1}}}
\fi
\NewDocumentCommand\Sym{mm}{\ensuremath{\operatorname{Sym}^{#1}\!\of{#2}}}
\NewDocumentCommand\GL{m}{\ensuremath{\operatorname{GL}_{#1}}}
\NewDocumentCommand\Lie{mo}{\ensuremath{\mathfrak{\MakeLowercase{#1}}\IfValueT{#2}{_{#2}}}}
\NewDocumentCommand\LieGL{m}{\Lie{GL}[#1]}
\NewDocumentCommand\SL{m}{\ensuremath{\operatorname{SL}_{#1}}}

\NewDocumentCommand\SU{m}{\ensuremath{\operatorname{SU}_{#1}}}

\NewDocumentCommand\nForms{omm}
{%
\IfValueTF{#1}
{\ensuremath{\vb{#1}^{#2}_{#3}}}
{\ensuremath{\Omega^{#2}_{#3}}}
}%
\NewDocumentCommand\Lm{smm}{\ensuremath{\Lambda^{#2}\IfBooleanTF{#1}{\!\left(#3\right)}{#3}}}
\NewDocumentCommand\cohomology{omm}{\ensuremath{H_{\IfValueT{#1}{\IfStrEq{#1}{db}{\bar\partial}{#1}}}^{#2}\of{#3}}}
\DeclareMathOperator{\Ad}{Ad}
\DeclareMathOperator{\ad}{ad}
\NewDocumentCommand\LieDer{}{\ensuremath{\mathcal L}}
\NewDocumentCommand\hook{}{\ensuremath{\mathbin{ \hbox{\vrule height1.4pt
        width4pt depth-1pt \vrule height4pt width0.4pt depth-1pt}}}}

\NewDocumentCommand\defeq{}{\coloneqq}

\NewDocumentCommand\smooth{}{\ensuremath{C^{\infty}}\xspace}
\NewDocumentCommand\MakeLie{m}{\expandafter\def\csname Lie#1\endcsname{\Lie{#1}}}
\def\lst{A,B,G,H,K,L,P,Q,R,S,T,U,Z}
\makeatletter
\@for\i:=\lst\do{\expandafter\MakeLie \i}
\makeatother
\NewDocumentCommand\lb{smm}%
{%
\IfBooleanTF{#1}%
{
\ensuremath{\left[{#2},{#3}\right]}
}
{
\ensuremath{[{#2}{#3}]}
}
}%
\NewDocumentCommand\prol{omo}{#2^{(\IfValueTF{#1}{#1}{1})\IfValueT{#3}{#3}}}
\NewDocumentCommand\SpC{m}{\cohomology{0,2}{\Lie{#1}}}

\NewDocumentCommand\At{o}{\operatorname{At}_{\IfValueT{#1}{#1}}}
\newcommand\varul[2][3]{\mkern#1mu\underline{\mkern-#1mu#2\mkern-#1mu}\mkern#1mu}
\NewDocumentCommand\qg{}{\varul[2]{γ}}
\NewDocumentCommand\qx{}{\varul[2]{x}}

\NewDocumentCommand\qE{}{\varul{E}}

\NewDocumentCommand{\Dlb}{sm}%
{%
\IfBooleanTF{#1}%
{
\Dlb{\left(#2\right)}%
}
{
#2^{\bar\partial}%
}
}%
\NewDocumentCommand{\Cch}{sm}%
{%
\IfBooleanTF{#1}%
{
\left(#2\right)^{\vee}%
}
{
\check{#2}%
}
}%
\NewDocumentCommand\Slovak{}{Slov\'ak}

\NewDocumentCommand\cpx{om}{\ensuremath{#2_{\bullet\IfValueT{#1}{-#1}}}}
\NewDocumentCommand\gLm{omom}%
{%
\IfValueTF{#1}%
{
{\varul[4]{#1}}%
}
{
\varul{\Lambda}%
}
^{#2,\IfValueTF{#3}{#3,#4}{#4}}%
}%

\NewDocumentCommand\gForms{omom}%
{%
\IfValueTF{#1}%
{
{\varul[4]{\vb{#1}}}%
}
{
\varul{\Omega}%
}
^{#2,\IfValueTF{#3}{#3,#4}{#4}}%
}%

\NewDocumentCommand\ch{mo}{\operatorname{ch}_{#1}\IfValueT{#2}{\!\left(#2\right)}}

\newcommand{\amal}[3]{\ensuremath{{#1} \mathbin{\times^{#2}} \! #3}}
\NewDocumentCommand\vb{m}{\ensuremath{\bm{#1}}}
\NewDocumentCommand\Conn{m}{\ensuremath{\mathscr{A}_{#1}}}
\makeatletter
\def\l@subsection{\@tocline{2}{0pt}{2.5pc}{5pc}{}}
\def\l@section{\@tocline{1}{0pt}{2.5pc}{5pc}{}}
\def\@tocline#1#2#3#4#5#6#7{\relax
  \ifnum #1>\c@tocdepth 
  \else
    \par \addpenalty\@secpenalty\addvspace{#2}%
    \begingroup \hyphenpenalty\@M
    \@ifempty{#4}{%
      \@tempdima\csname r@tocindent\number#1\endcsname\relax
    }{%
      \@tempdima#4\relax
    }%
    \parindent\z@\relax
\leftskip#3\relax \advance\leftskip\@tempdima\relax
    #5\leavevmode\hskip-\@tempdima{#6}\nobreak\relax
    ,~#7\par
    \nobreak
    \endgroup
  \fi}
\makeatother
\NewDocumentCommand\CS{mo}{T_{#1\IfValueT{#2}{,#2}}}
\NewDocumentCommand\proj{om}{\operatorname{proj}_{\IfValueT{#1}{#1}}\!\left(#2\right)}
\NewDocumentCommand\exactSeq{smO{}mO{}m}%
{%
\begin{tikzcd}[cramped, sep=small,ampersand replacement=\&] 
\IfBooleanTF{#1}{1}{0} \rar \& #2 \rar{#3} \& #4 \rar{#5} \& #6 \rar \& \IfBooleanTF{#1}{1}{0} 
\end{tikzcd}%
}%
\NewDocumentCommand\bundle{smO{}mO{}m}%
{%
\IfBooleanTF{#1}%
{%
\begin{tikzcd}[cramped, sep=small,ampersand replacement=\&] 
#2 \rar{#3} \& #4 \rar{#5} \& #6 
\end{tikzcd}%
}%
{%
\begin{tikzcd}[cramped, sep=small,ampersand replacement=\&] 
#2 \rar{#3} \& #4 \dar{#5} \\ \& #6 
\end{tikzcd}%
}%
}%
\ifdefined\DeclareUnicodeCharacter
  \DeclareUnicodeCharacter{0393}{\Gamma}         
  \DeclareUnicodeCharacter{0394}{\Delta}         
  \DeclareUnicodeCharacter{0398}{\Theta}         
  \DeclareUnicodeCharacter{039B}{\Lambda}        
  \DeclareUnicodeCharacter{039E}{\Xi}            
  \DeclareUnicodeCharacter{03A0}{\Pi}            
  \DeclareUnicodeCharacter{03A3}{\Sigma}         
  \DeclareUnicodeCharacter{03A5}{\Upsilon}       
  \DeclareUnicodeCharacter{03A6}{\Phi}           
  \DeclareUnicodeCharacter{03A8}{\Psi}           
  \DeclareUnicodeCharacter{03A9}{\Omega}         
  \DeclareUnicodeCharacter{03B1}{\alpha}         
  \DeclareUnicodeCharacter{03B2}{\beta}          
  \DeclareUnicodeCharacter{03B3}{\gamma}         
  \DeclareUnicodeCharacter{03B4}{\delta}         
  \DeclareUnicodeCharacter{03B5}{\varepsilon}    
  \DeclareUnicodeCharacter{03B6}{\zeta}          
  \DeclareUnicodeCharacter{03B7}{\eta}           
  \DeclareUnicodeCharacter{03B8}{\theta}         
  \DeclareUnicodeCharacter{03B9}{\iota}          
  \DeclareUnicodeCharacter{03BA}{\kappa}         
  \DeclareUnicodeCharacter{03BB}{\lambda}        
  \DeclareUnicodeCharacter{03BC}{\mu}            
  \DeclareUnicodeCharacter{03BD}{\nu}            
  \DeclareUnicodeCharacter{03BE}{\xi}            
  \DeclareUnicodeCharacter{03C0}{\pi}            
  \DeclareUnicodeCharacter{03C1}{\rho}           
  \DeclareUnicodeCharacter{03C2}{\varsigma}      
  \DeclareUnicodeCharacter{03C3}{\sigma}         
  \DeclareUnicodeCharacter{03C4}{\tau}           
  \DeclareUnicodeCharacter{03C5}{\upsilon}       
  \DeclareUnicodeCharacter{03C6}{\varphi}        
  \DeclareUnicodeCharacter{03C7}{\chi_}           
  \DeclareUnicodeCharacter{03C8}{\psi}           
  \DeclareUnicodeCharacter{03C9}{\omega_}         
  \DeclareUnicodeCharacter{03D1}{\thetasymbol}   
  \DeclareUnicodeCharacter{03D5}{\phisymbol}     
  \DeclareUnicodeCharacter{03D6}{\varpi_}      
  \DeclareUnicodeCharacter{03DD}{\digamma}       
  \DeclareUnicodeCharacter{03F1}{\rhosymbol}     
  \DeclareUnicodeCharacter{03F5}{\epsilonsymbol} 
\fi
\begin{document}
\maketitle
\begin{abstract}%
Characteristic class relations in Dolbeault cohomology follow from the existence of a holomorphic geometric structure (for example, holomorphic conformal structures, holomorphic Engel distributions, holomorphic projective connections, and holomorphic foliations).
These relations can be calculated directly from the representation theory of the structure group, without selecting any metric or connection or having any knowledge of the Dolbeault cohomology groups of the manifold.
This paper improves on its predecessor \cite{mckay2022} by allowing infinite type geometric structures.%
\end{abstract}
\tableofcontents
\section{Introduction}
We explain how to compute equations on Chern classes and Chern--Simons invariants of various holomorphic geometric structures on complex manifolds.
Applied to holomorphic foliations, for example, our computational recipe trivially yields the Baum--Bott theorem.
\section{Holomorphic geometric structures}
We need the notation, so we will define \(G\)-structures \cite{Gardner:1989}.
\subsection{Notation}
Denote the Lie algebra of a Lie group \(G\) as \(\LieG\), and similarly denote the Lie algebra of any Lie group by the corresponding fraktur font expression.
All Lie groups are complex analytic and finite dimensional.
All \(G\)-modules are finite dimensional and holomorphic.
Denote the left invariant Maurer--Cartan \(1\)-form on \(G\) as \(g^{-1}dg\).
Take a holomorphic right principal bundle \bundle*{G}{E}{M}.
For each vector \(v \in \LieG\), denote also by \(v\) the associated vector field on \(E\): for any \(x \in E\),
\[
v(x) = \left.\frac{d}{dt}\right|_{t=0} x \, e^{tv}.
\]
If we wish to be more precise, we denote the vector field \(v\) on \(E\) as \(v_E\).
For any \(G\)-action on a manifold \(X\), denote by \(\amal{E}{G}{X}\) the quotient of \(E \times X\) by the diagonal right \(G\)-action \((e,x)h=(eh,h^{-1}x)\).
If \(V\) is a complex analytic \(G\)-module, denote the associated vector bundle by \(\vb{V} \defeq \amal{E}{G}{V}\).
\subsection{\texorpdfstring{Infinitesimal theory of $G$-structures}{Infinitesimal theory of G-structures}}
Let us recall the usual representation theory associated to \(G\)-structures \cite{Gardner:1989}, \cite{Ivey/Landsberg:2016} chapter 8, \cite{Sternberg:1983}, in the complex analytic setting.
Suppose that \(G\) is a complex Lie group and \(V\) a finite dimensional holomorphic \(G\)-module.
For each \(a\otimes\xi\in\LieG\otimes V^*\), define \(\rho_{a\otimes \xi}\in V^*\otimes V^*\otimes V\) by
\[
\rho_{a\otimes \xi}\colon x\otimes y\in V\otimes V\mapsto \rho_{a\otimes\xi}(x,y)\defeq\xi(y)a(x)\in V
\]
and extend by complex linearity:
\[
\rho\colon \LieG\otimes V^*\to V^*\otimes V^*\otimes V.
\]
Define
\[
\delta\colon \LieG\otimes V^*\to V\otimes\Lm{2}{V}^*,
\]
by, for \(x,y\in V\), \(a\in\LieG\) and \(\xi\in V^*\),
\[
\delta_{\xi\otimes a}(x,y)=\xi(y)a(x)-\xi(x)a(y).
\]
Let \(\prol\LieG\subseteq\LieG\otimes V^*\) be the kernel of \(\delta\).
So elements of \(\prol\LieG\) are precisely elements of \(\LieG\otimes V^*\) mapped by \(\rho\) to \(V\otimes\Sym{2}{V}^*\).
We get \(\LieG\otimes V^*\) 
to act on \(V_1\defeq V\oplus\LieG\) by
\[
(a\otimes\xi)(x,b)=(x,b+\xi(x)a).
\]
The \emph{prolongation} of the \(G\)-module \(V\) is the \(G_1\)-module \(V_1\) where \(G_1\defeq G\rtimes\prol\LieG\) acting on \(V_1\) by usual action of \(G\) and by this action of \(\prol\LieG\subseteq\LieG\otimes V^*\).
Define the Spencer cohomology \(\SpC{G}\) by the exact sequence of \(G\)-modules
\[
\begin{tikzcd}
0\arrow{r}&
\prol\LieG\arrow{r}&
\LieG\otimes V^*\arrow{r}{\delta}&
V\otimes\Lm{2}{V}^*\arrow{r}{[]}&
\SpC{G}\arrow{r}&
0.
\end{tikzcd}
\]
The abelian group \(\prol\LieG\subset\LieG\otimes V^*\) acts on \(V\oplus\LieG\) by
\[
Q\in\prol\LieG,(v,A)\in V\oplus\LieG\mapsto (v,Qv+A).
\]
Let \(g\in G\) act on \(Q\in\Sym{2}{V}^*\otimes V\) by
\[
(gQ)(u,v)\defeq g(Q(g^{-1}u,g^{-1}v)),
\]
Form the semidirect product \(G\rtimes\prol\LieG\) by 
\[
(g_1,Q_1)(g_2,Q_2)\defeq(g_1g_2,Q_1+g_1Q_2).
\]

\subsection{\texorpdfstring{$G$-structures}{G-structures}}
Pick a finite dimensional complex vector space \(V\) and a complex manifold \(M\) of dimension equal to that of \(V\).
The \(V\)-valued \emph{frame bundle} of \(M\) is the set \(FM\) of all pairs \((m,u)\) of point \(m\in M\) and complex linear isomorphism \(u\colon T_m M\to V\).
Let \(\pi\colon (m,u)\in FM\to m\in M\).
The group \(\GL{V}\) acts on \(FM\) by the right action \((m,u)g=(m,g^{-1}u)\), also denoted \(r_g(m,u)\).
Clearly \(FM\) is a holomorphic principal right \(\GL{V}\)-bundle.
The \emph{soldering form} \(σ\) is the \(V\)-valued differential form whose value on a vector \(v\in T_{(m,u)} FM\) is \(v\hook σ=u(\pi'_{(m,u)}v)\), so that \(r_g^*σ=\Ad_g^{-1}σ\) for \(g\in\GL{V}\).
Differentiating, \(\LieDer_v σ = -[v,σ]\) for \(v\in\LieGL{V}\).
By the Cartan formula, \(\LieDer_v σ = v\hook dσ + d(v\hook σ)=v\hook dσ\).

Suppose that \(G\) is a complex group and that \(V\) is a \(G\)-module with representation \(\rho_V\colon G\to\GL{V}\).
We also denote by \(\rho_V\) the associated Lie algebra morphism \(\LieG\to\LieGL{V}\).
A \emph{\(G\)-structure} on \(M\) is a holomorphic right principal \(G\)-bundle \(E\to M\) together with a \(G\)-equivariant holomorphic bundle map \(E\to FM\).
If \(\rho_V\) is an embedding, \(G\)-structures can also be described as holomorphic sections of \(FM/G\to M\).

A \emph{connection covector} at a point \(e_0\in E\) is a covector \(γ\in T^*_{e_0} E\otimes\LieG\) so that \(v\hook γ=v\) for \(v\in\LieG\).
From the above,
\[
dσ=-γ\wedge σ+tσ^2
\]
for some \(t\in \Lm{2}{V}^*\otimes V\), the \emph{torsion} of the pseudoconnection, where \(\sigma^2\) means
\[
(σ^2)^{ij}=\frac{1}{2}σ^i\wedge σ^j.
\]
Any two connection covectors \(γ,γ'\) agree up to \(γ'=γ+Aσ\) where \(A\in V^*\otimes\LieG\).
The difference in torsion is \(t'=t+\delta_A\).
Therefore the projection \(T\) of \(t\) to Spencer cohomology, the \emph{torsion} of the \(G\)-structure at \(e_0\in E\), is defined independently of the choice of connection covector.
Clearly \(T\) is \(G\)-equivariant, giving a section \(T\) of the \emph{torsion bundle}: the holomorphic vector bundle \(\amal{E}{G}{\SpC{G}}\).

An \emph{anchor} for the \(G\)-structure is a section of the associated vector bundle
\[
\vb{V}\otimes\Lm{2}{\vb{V}}^*=\amal{E}{G}{\left(V\otimes\Lm{2}{V}^*\right)},
\]
which lifts the intrinsic torsion, i.e. has image \(T\) in Spencer cohomology; if an anchor exists, the \(G\)-structure is \emph{anchored} or \emph{prolongs}.
An anchor exists, for example, if that associated vector bundle has trivial first cohomology, or if the intrinsic torsion vanishes (so we can use \(0\) as anchor), or if \(G\) is reductive, since we can then lift every morphism of \(G\)-modules.
For each anchor \(t\), the \emph{associated prolongation} \(\prol{E}=\prol{E}_t\) of the \(G\)-structure consists of the set of triples \((m,u,γ)\) of point \((m,u)\in FM\) and connection covector \(γ\) at that point whose torsion agrees with the anchor \(t\).
Thinking of \(\prol\LieG\) as a complex Lie group under addition, the prolongation \(\prol{E}\) is a principal right \(G\rtimes\prol\LieG\)-bundle over \(M\) and a principal right \(\prol\LieG\)-bundle over \(E\); denote the bundle maps as \(\pi_E\colon\prol{E}\to E\) and \(\pi_M\colon\prol{E}\to M\).
Pullback the soldering form \(σ\) from \(E\), and calling it by the same name.
Define the prolongation \(1\)-form \(γ\) on \(E\) by
\[
v\hook γ_{(m_0,u_0,γ_0)}=\pi_E'(v)\hook γ_0.
\]
Note that \(\prol{E}\to E\) is a \(\prol\LieG\)-structure, the \emph{prolongation} of \(E\), mapped to the frame bundle of \(E\) by
\[
(m_0,u_0,γ_0)\mapsto \left(u_0\circ\pi_{E}'(m_0,u_0),γ_0\right).
\]
The right action of \(\prol\LieG\) on \(\prol{E}\) is
\[
(m_0,u_0,γ_0)Q\defeq (m_0,u_0,γ_0-Qσ).
\]
Under this action,
\begin{align*}
r_Q^*σ&=σ,\\
r_Q^*γ&=γ-Qσ.
\end{align*}
The group \(G\) also acts on the right on \(\prol{E}\), so that the bundle map \(\prol{E}\to E\) is equivariant, by
\[
(m_0,u_0,γ_0)g\defeq
(m_0,g^{-1}u_0,\Ad_g^{-1}\left(γ_0(r_g^{-1})'\right),
\]
giving an action of the semidirect product \(G\rtimes\prol\LieG\).
Under this action
\begin{align*}
r_{(g,Q)}^*σ&=g^{-1}σ,\\
r_{(g,Q)}^*γ&=\Ad_g^{-1}γ-Qg^{-1}σ.
\end{align*}
We can cover \(\prol{E}\) in open sets \(\prol{E}_a\) on each of which we can pick some \emph{prolongation pseudoconnection form}, a (holomorphic) \((1,0)\)-form \(ϖa\) on \(\prol{E}_a\) valued in \(\prol\LieG\), so that for \((v,q)\) in the Lie algebra of \(G\rtimes\prol\LieG\), \((v,q)\hook ϖa=q\).
Our forms \(σ,γ\) then satisfy Cartan's structure equations: 
\begin{align*}
dσ+γ\wedge σ&=tσ^2,\\
dγ+\frac{1}{2}\lb{γ}{γ}+ϖa\wedge σ&=k_aσ^2,
\end{align*}
where \(k_a\colon\prol{E}_a\to\LieG\otimes\Lm{2}{V}^*\) is the \emph{curvature}. 
On the overlaps \(E_{ab}\defeq E_a\cap E_b\), \(ϖb-ϖa=p_{ab}σ\) for a unique \(C^{\infty}\) (holomorphic) map \(p_{ab}\colon E_{ab}\to V^*\otimes\prol\LieG\).
\subsection{Langlands decomposition}
A \emph{Langlands decomposition} of a complex Lie group \(G\) is a semidirect product decomposition \(G = G_0 \ltimes G_+\) in closed complex subgroups, where \(G_+\) is a connected and simply connected solvable complex Lie group and \(G_0\) is a reductive complex linear algebraic group.
For example:
\begin{enumerate}
\item
This definition generalizes the usual Langlands decomposition of any parabolic subgroups of any complex semisimple Lie group \cite{Knapp:2002} p. 481.
\item
Every connected and simply connected complex Lie group \(G\) admits a Langlands decomposition in which \(G_0\) is a maximal semisimple subgroup and \(G_+\) is the solvradical
\cite{Varadarajan:1984} p. 244 theorem 3.18.13.
\item
Any connected complex Lie group \(G\) admits a faithful holomorphic representation just when it admits a Langlands decomposition in which \(G_0\) is a complex linearly reductive group and \(G_+\) is the nilradical
\cite{Hilgert.Neeb:2012} p. 595 theorem 16.2.7.
\item
Every complex linear algebraic group \(G\) (perhaps disconnected) admits a Langlands decomposition in which \(G_0 \subset G\) is a maximal reductive subgroup and \(G_+ \subset G\) is the unipotent radical \cite{Hochschild:1981} p. 117 theorem 4.3.
In all of our examples below, \(G\) will be complex linear algebraic.
\end{enumerate}
Every connected and simply connected solvable complex Lie group \(G_+\) is biholomorphic to complex affine space \cite{Hilgert.Neeb:2012} p. 543 theorem 14.3.8, and so is a contractible Stein manifold. 

A \emph{Langlands decomposition} of a filtered \(G\)-module \(V\) is a Langlands decomposition \(G\rtimes\prol\LieG=G_0\rtimes G_+\) so that \(G_0\subseteq G\) and \(G_+\) acts trivially on the associated graded.

\subsection{Infinitesimal characteristic forms}\label{subsec:inf.char}
Take a complex Lie group \(G\) and \(G\)-module \(V\) with a Langlands decomposition, and an \(\LieG_0\)-module \(W\), with Lie algebra action \(\rho_W \colon \LieG_0 \to \LieGL{W}\).
Each \(x=a\otimes\xi\in\LieG\otimes V^*\) and \(y\in V\) has associated \(x(y)=\xi(y)a\in \LieG\), contracting on\(V^*,V\).
Since \(\prol\LieG\subseteq\LieG\otimes V^*\), we can take any \(x\in\prol\LieG\) and \(y\in V^*\) and this defines \(x(y)\in\LieG\).
Denote by \(\proj[\LieG_0]{x(y)}\) the projection of this element of \(\LieG\) to \(\LieG_0\) by the Langlands decomposition.
The \emph{Atiyah form} \(a=a_W\) is the element
\[
a \in{\prol\LieG}^* \otimes V^* \otimes \LieGL{W}
\]
given by 
\[
a(x,y)=-\rho_W \circ \proj[\LieG_0]{x(y)},
\]
for \(x\in\prol\LieG, y\in V\).
If \(W\) is not specified, we take \(W\defeq\LieG_0\).
(Note that we do \emph{not} require that \(W\) be a \(G_0\)-module, so there might not be an associated vector bundle for a \(G\)-structure.)
The \emph{Chern forms} \(c_k\) are
\[
c_k \in \Sym{k}{\prol\LieG \otimes V}^*
\]
given by
\[
\det\left(I+\frac{ia}{2\pi}\right)=1+c_1+c_2+\dots=c.
\]
Analogously define the Chern character forms and Todd forms.
More generally, if \(f\) is an \(\LieG_0\)-invariant complex symmetric multilinear form on \(\LieG_0\), say of degree \(k\), we associate to \(f\) the element, denoted by the same name, 
\[
f \in \Sym{k}{\prol\LieG \otimes V}^*
\]
given by \(f(a,\dots,a)\).
The \emph{Chern--Simons form} of \(f\) is
\[
\CS{f}\of{u,v,w_+,w_-}\\
\defeq \sum_{j=0}^{k-1} a_j 
f(
	u,
	\underbrace{v,\dots,v}_{j},
	\underbrace{a\of{w_+,w_-},\dots,a\of{w_+,w_-}}_{k-j-1}%
)
\]
where
\[
a_j\defeq\frac{(-1)^j(k-1)!}{
(k+j)!(k-1-j)!}
\]
and
\begin{align*}
u,v&\in\LieG_0, \\
w_+&\in\prol\LieG,\\
w_-&\in V.
\end{align*}
(N.B. the expression for \(a_j\) is not the same as in the paper of Chern and Simons \cite{Chern/Simons:1974}; their \(A_j\) is \(A_j=a_j/2^j\).)
The splitting principle: if \(0 \to U \to V \to W \to 0\) is an exact sequence of \(\LieG_0\)-modules, extend a basis of \(U\) into a basis of \(V\), so 
\[
a_V=
\begin{pmatrix}
a_U & * \\
0 & a_W
\end{pmatrix}
\]
and compute the determinant: \(c(U)c(W)=c(V)\).
The \emph{tangent bundle Atiyah form} is
\[
a_T \colon x \in \prol\LieG, y,z \in V \mapsto 
\frac{a(x,y)z+a(x,z)y}{2}.
\]

\section{Characteristic classes}
\subsection{The connection bundle}\label{subsection:connection.bundle}
We review some well known material to establish notation and terminology, following the standard references \cite{Atiyah1957,Chern/Simons:1974} .
Take a holomorphic right principal bundle
\[
\bundle{G}{E}[\pi]{M}.
\]
Let \(\ad_E\defeq \amal{E}{G}{\LieG}\).
The \(G\)-invariant exact sequence
\[
\exactSeq{\ker \pi'}{TE}[\pi']{\pi^* TM}
\]
of vector bundles on \(E\) quotients by \(G\)-action to an exact sequence of vector bundles on \(M\):
\[
\exactSeq{\ad_E}{\At[E]}{TM}
\]
with middle term \(\At[E]\) the \emph{Atiyah bundle}.
A holomorphic (\(C^{\infty}\)) splitting \(s\) of this exact sequence determines and is determined by a holomorphic (\(C^{\infty}\)) \((1,0)\)-connection \(ω{}=ωs\) for the bundle \(E \to M\).
The connection form is the unique \((1,0)\)-form so that \(s(v)\hook ω{}=0\) and \(w\hook ω{}=w\) for \(w\in\LieG\), i.e. the splitting lifts each tangent vector to its horizontal lift \cite{Atiyah1957}.
Write the section as \(s=s_{ω{}}\).
The \emph{connection bundle} of \(E\) is the affine subbundle \(\Conn{E}\subset T^*M \otimes_M \At[E]\) consisting of complex linear maps which split the sequence over some point of \(M\).
So holomorphic \((C^{\infty})\) \((1,0)\)-connections are precisely holomorphic \((C^{\infty})\) sections of the connection bundle.
Differences of two connections lie in \(T^*M \otimes \ad_E\).
So \(\Conn{E} \to M\) is a holomorphic bundle of affine spaces, modelled on the vector bundle \(T^*M \otimes_M \ad_E\).
Each element \(v\in\At[E,m]\) is an \(G\)-invariant section of \(\left.TE\right|_{E_m}\to\left.\pi^*TM\right|_{E_m}\).
The holomorphic (\(C^{\infty}\)) sections of \(\Conn{E}\) are precisely the holomorphic (\(C^{\infty}\)) \((1,0)\)-connections. 
Each fiber \(\Conn{E,m}\) is the set of all \(G\)-invariant sections \(ω{}\) of
\[
\left.T^*E \otimes \LieG\right|_{E_m}
\]
so that \(v\hook ω{}=v\) for \(v\in\LieG\) with \(G\)-invariance:
\[
ω{eg}=\Ad_g^{-1} r_g^{-1*}ωe
\]
for \(g\in G\).

Denote the bundle map as \(\delta \colon \Conn{E} \to M\), with pullback
\[
\begin{tikzcd}
E\times_M\Conn{E} \rar{\Delta} \dar & E \dar \\
\Conn{E} \rar{\delta} & M.
\end{tikzcd}
\]
Each point \(x \in E\times_M\Conn{E}\) has the form \(x=(m_0,ω0,e_0)\) for some \(m_0 \in M\), \(e_0 \in E_m\), \(ω0\colon T_e E \to \LieH\) so that \(w\hook ω0=w\) for \(w \in \LieG\).
There is a holomorphic connection \(ω{}\) on \(E\times_M\Conn{E}\) defined for a tangent vector \(v \in T_x E\times_M\Conn{E}\) by \(v\hook ω{}=(\Delta'(x)v)\hook ω0\) \cite{Biswas:2017}.
Given a holomorphic \((C^{\infty})\) \((1,0)\)-connection \(ω0\) on \(E \to M\), map \(\Phi \colon e \in E \to \Phi(e)\defeq(m,γ_0,e)\in E\times_M\Conn{E}\) and compose with the bundle map \(\pi\colon E\times_M\Conn{E}\to\Conn{E}\) to get a section of the connection bundle.
Pullback the bundle \(E\times_M\Conn{E}\) by the section to get a map \(\varphi \colon E \to \Conn{E}\), so that \(\varphi^*E\times_M\Conn{E}=E\) has pullback connection \(\varphi^*ω{}=ω0\).

The connection bundle of a vector bundle is the connection bundle of its associated principal bundle.

\subsection{\texorpdfstring{Connection bundles of $G$-structures}{Connection bundles of G-structures}}\label{subsection:ConnBundlesShapes}
Take a complex Lie group \(G\), finite dimensional holomorphic \(G\)-module \(V\), and a Langlands decomposition \(G\rtimes\prol\LieG=G_0\rtimes G_+\).
Take a holomorphic anchored \(G\)-structure \(G\to E\to M\) with prolongation \(\prol{E}\to E\to M\).
Each point \(x_1\in\prol{E}\) has the form \(x_1=(m_1,u_1,γ_1)\) where \(γ_1\) is a connection covector 
\[
γ_1\colon T_{(m_1,u_1)}E\to\LieG.
\]
Write the associated point \(G_+x_1\in\qE\defeq\prol{E}/G_+\) as \(\qx_1\defeq G_+x_1\).
Take a \(G_0\)-equivariant projection \(q\colon\LieG\to\LieG_0\).
The covector \(γ_1\colon T_{(m_1,u_1)} E\to\LieG\) gives a covector \(γ_{x_1}\defeq γ_1\pi_{E}'(m_1,u_1)\colon T_{x_1}\prol{E}\to\LieG\), hence a \(\LieG\)-valued \(1\)-form \(γ\) on \(\prol{E}\).
The covector \(qγ_{x_1}\) vanishes on the fibers of \(\pi_{\qE}\colon x_1\in\prol{E}\to\qx_1\in\qE\defeq\prol{E}/G_+\), i.e. is semibasic, so determines a unique covector \(\qg_{x_1}\colon T_{\qx_1} \qE\to\LieG_0\), uniquely defined by \(\pi_{\qE}^*(\qg_{x_1})=qγ_{x_1}\pi_E'\), hence a connection covector.
Map
\[
\Phi \colon x_1=(m_1,u_1,γ_1)\in \prol{E} 
\mapsto 
(m_1,\qg_{x_1},\qx_1)\in\qE\times_M\Conn{\qE}
\]
which we quotient by \(G_0\)-action to get
\[
\begin{tikzcd}
\prol{E} \arrow{d} \arrow{r}{\Phi} & \qE\times_M\Conn{\qE} \arrow{d} \arrow{r}{\Delta} & \qE \arrow{d} \\
\prol{E}/G_0 \arrow{r}{\phi} & \Conn{\qE} \arrow{r}{\delta} & M
\end{tikzcd}
\]
applying the commutative diagram of~\vref{subsection:connection.bundle} but to \(\qE\) instead of \(E\).
\begin{lemma}
\(\Phi^*ω{}=qγ\).
\end{lemma}
\begin{proof}
For \(m_1 \in M\), \(x_1=(m_1,u_1,γ_1)\in \prol{E}_m\), \(v_1\in T_{x_1}\prol{E}\), let
\[
y_1\defeq(m_1,\qg_{x_1},\qx_1)=\Phi(x_1)\in \qE\times_M\Conn{\qE}.
\]
so \(\Delta(y_1)=\Delta(m_1,\qg_{x_1},\qx_1)=\qx_1\).

Compute
\begin{align*}
v_1 \hook (\Phi^*ω{})_{x_1}
&=
\Phi'(x_1)v_1 \hook ω{\Phi(x_1)},
\\
&=
\Phi'(x_1)v_1 \hook ω{y_1},
\\
&=
(\Delta'(y_1)\Phi'(x_1)v_1) \hook qγ_1,
\\
&=
(\Delta\Phi)'(x_1)v_1\hook qγ_1,
\\
&=\pi_{\qE}'(x_1)v_1\hook qγ_1,
\\
&=v_1\hook\pi_{\qE}^*qγ_1,
\\
&=v_1\hook qγ_1\pi_E',
\\
&=v_1\hook qγ_{x_1}.
\end{align*}
\end{proof}
We take advantage of this and write \(ω{}\) to mean \(qγ\) henceforth.
Since \(ω{}\) is a holomorphic connection on \(\qE\times_M\Conn{\qE}\), its curvature is \(\Omega\defeq dω{}+\frac{1}{2}\lb{ω{}}{ω{}}\), and pulls back to a form we also denote \(\Omega\), \(\Omega\defeq q(dγ+\frac{1}{2}\lb{γ}{γ})\) on \(\prol{E}\), even though \(γ\) is not a connection on \(\prol{E}\).
The Bianchi identity \(d\Omega=\lb{\Omega}{ω{}}\) on \(\qE\times_M\Conn{\qE}\) ensures the same identity on \(\prol{E}\), even though \(γ\) is not a connection on \(\prol{E}\) and \(\Omega\) is not the curvature of a connection.

Similarly, for any \(G_0\)-invariant complex polynomial function \(f \colon \LieG_0 \to \C{}\), thought of as a symmetric multilinear form, the expression
\[
f_E\defeq f(\Omega,\dots,\Omega)
\]
on \(\prol{E}\) is the pullback of the Chern form \(f_{\qE\times_M\Conn{\qE}}\) for the connection \(ω{}\) on the bundle \(\qE\times_M\Conn{\qE} \to \Conn{\qE}\).
In particular, \(f_E\) is a closed holomorphic differential form.
From Cartan's structure equations
\[
\Omega=q(-ϖa\wedge σ+kσ^2).
\]
Similarly define 
\[
\CS{f}[E]\defeq \CS{f}(ω{},\lb{ω{}}{ω{}},ϖ{},σ),
\]
which is the pullback of the Chern--Simons form \(\CS{f}[\qE\times_M\Conn{\qE}]\), hence \(d\CS{f}[E]=f_E\).

\subsection{Smooth reduction of structure group}
Take a complex Lie group \(G\), finite dimensional holomorphic \(G\)-module \(V\), and a Langlands decomposition \(G\rtimes\prol\LieG=G_0\rtimes G_+\).
Take a holomorphic anchored \(G\)-structure \(G\to E\to M\) with prolongation \(\prol{E}\to E\to M\).

Since \(G\rtimes\prol\LieG/G_0\) is contractible, \(\prol{E}/G_0\to M\) admits a \smooth section \(s \colon M \to \prol{E}/G_0\) i.e. a \smooth \(G_0\)-reduction of structure group.
The \(1\)-form \(γ\) on \(\prol{E}\) pulls back to a \(1\)-form \(γ\) on \(s^*\prol{E}\).
Let \(\qE\defeq\prol{E}/G_+\), a holomorphic principal right \(G_0\)-bundle \(\bundle*{G_0}{\qE}{M}\).
So this \(1\)-form \(γ\) extends from \(s^*\prol{E}\) to a unique \(1\)-form on \(\qE\cong s^*\prol{E}\) which we also denote \(γ\), and which satisfies \(v \hook γ=v\) for \(v \in \LieG_0\).
\begin{lemma}
The \(1\)-form \(γ\) on \(\qE\) associated to any \smooth (or holomorphic) \(G_0\)-reduction is a \smooth (holomorphic) \((1,0)\)-connection \(1\)-form.
\end{lemma}
\begin{proof}
Pick a point \((m_0,u_0,γ_0)\in s^*\prol{E}\), i.e. with \(m_0 \in M\) and \(u_0\in E_{m_0}\) and \((m_0,u_0)G_0=s(m_0)\).
So \(γ\) at the corresponding point of \(\qE\) is the \(1\)-form which pulls back by \(\prol{E}\to\qE\) to become \(γ\) at the point \((m_0,u_0,γ_0)\).
If we replace \((m_0,u_0,γ_0)\) by some point \((m_0,u_0,γ_0)(g,Q)\), for some \((g,Q)\in G\rtimes\prol\LieG\), 
\[
r_{(g,Q)}^*γ = \Ad_g^{-1}γ-Qg^{-1}σ.
\]
In particular, for \((g,Q)=(g_+,0)=g_+\in G_+\),
\[
r_{g_+}^*γ = \Ad_{g_+}^{-1}γ.
\]
\end{proof}

\subsection{Characteristic forms and classes}
The \emph{Atiyah form}, \(k^{\text{th}}\) \emph{Chern form}, \emph{Chern character form}, \emph{Todd form}, etc.\ of an anchored \(G\)-structure with a Langlands decomposition is the form identified with the infinitesimal Atiyah form, \(k^{\text{th}}\) Chern form, Chern character form, Todd form, etc.\ of for any \(G_0\)-invariant homogeneous polynomial function \(f \colon \LieH_0 \to \C\) when plugging in the forms \(γ\) and \(ϖa\wedge σ\):
\[
f(ϖ{},σ,\dots,ϖ{},σ)
\]
or, for a Chern--Simons form,
\[
\CS{f}\of{γ,γ,ϖ{},σ}.
\]
\begin{lemma}\label{lemma:main}
The Atiyah class, \(k^{\text{th}}\) Chern class, and so on, in Dolbeault cohomology of the bundle \(\qE \to M\) of a \(G\)-structure \(E \to M\) with Langlands decomposition is the class of the \((1,1)\)-part, \((k,k)\)-part, and so on, of the pullback by a \(C^{\infty}\) section \(s\) of the Atiyah form over each open set \(\prol{E}_a\).
The total Chern class, Chern character, Todd class, and so on, in Dolbeault cohomology of the bundle \(\qE \to M\) is the class of the pullback of the total Chern form, and so on.
\end{lemma}
\begin{proof}
Denote by \(w \in\LieG\mapsto \proj[\LieG_0]{w}\in\LieG_0\) some complex linear \(G_0\)-invariant projection, letting \(γ_0\defeq\proj[\LieG_0]{γ}\).
The Atiyah class \cite{Atiyah1957} of \(\qE\) is represented by
\[
a(M,\qE)=[\bar\partial γ_0]=\left[(dγ_0)^{1,1}\right].
\]
Pick local \smooth prolongation pseudoconnection \(1\)-forms \(ϖa\).
The \(2\)-form \((dγ_0)^{1,1}\) pulls back to \(s^*\prol{E}_a\) to
\[
-\proj[\LieG_0]{ϖa^{0,1}\wedge σ},
\]
noting that the curvature terms \(kσ^2\) are \((2,0)\)-forms, so make no contribution to this \((1,1)\)-form.
Also note that when we change from \(ϖa\) to \(ϖb\), the difference \(p_{ab}σ^2\) is also a \((2,0)\)-form, so makes no contribution to this \((1,1)\)-form.
The Atiyah class is represented by
\[
(dγ_0)^{1,1} = -\proj[\LieG_0]{ϖa^{0,1} \wedge σ} 
\in
\nForms[W]{1,1}{M}
\]
where \(W\subset\LieG_0\) is the projection to \(\LieG_0\) of the span of \(\LieG_+ V\)
\end{proof}
\begin{example}
If the infinitesimal first Chern form of a \(G\)-module vanishes, then every complex manifold with a \(G\)-structure modelled on that \(G\)-module has a holomorphic connection on its canonical bundle.
This happens, for instance, for holomorphic symplectic structures, as the symplectic group preserves the Liouville volume form.
\end{example}
\begin{example}
The Cartan geometries that arose in the two previous papers in this series \cite{McKay:2011,mckay2022} could also be described in the language of \(G\)-structures, recovering our previous theorems for them.
\end{example}
For any \(\LieG_0\)-module \(W\), even if \(W\) is \emph{not} a \(G_0\)-module, we still write \(a(M,\vb{W})\) to mean the class in Dolbeault cohomology associated to \(\rho_W \circ a\), even though \(\vb{W}\) does not exist.
We say that \(\vb{W}\) is a \emph{ghost vector bundle}.
\begin{corollary}
For a complex manifold admitting a \(G\)-structure, the Atiyah class of the tangent bundle is the Dolbeault class of the \((1,1)\)-part of the pullback by any \(C^{\infty}\) section of the form identified by a local Cartan connection with the tangent bundle Atiyah form.
\end{corollary}
The symmetry of the tangent bundle Atiyah class is a consequence of the well known symmetry of the Atiyah class of the tangent bundle.

\section{Example: Engel plane fields}
An Engel plane field is a holomorphic rank \(2\) subbundle \(\vb{W} \subset TM\) of the tangent bundle of a complex 4-fold \(M\) so that, near each point, there are local holomorphic sections \(u,v\) of \(\vb{W}\) so that \(u,v,\lb{u}{v},\lb{u}{\lb{u}{v}}\) are linearly independent tangent vector fields.
For more information, see \cite{BCGGG:1991} p. 50 Theorem II.5.1, \cite{Bryant/Hsu:1993}, \cite{Presas/Sola.Conde:2014}.
It is easy (essentially following the proof of \cite{BCGGG:1991} Theorem II.5.1) to see that the method of equivalence yields structure equations
\[
d
\begin{pmatrix}
σ^1 \\
σ^2 \\
σ^3 \\
σ^4
\end{pmatrix}
=
-
\begin{pmatrix}
2 γ_3^3+γ_4^4 & 0 & 0 & 0 \\
γ_1^2 & γ_3^3+γ_4^4 & 0 & 0 \\ 
γ_1^3 & γ_2^3 & γ_3^3 & 0 \\
γ_1^4 & γ_2^4 & γ^4_3 & γ^4_4 
\end{pmatrix}
\wedge
\begin{pmatrix}
σ^1 \\
σ^2 \\
σ^3 \\
σ^4
\end{pmatrix}
-
\begin{pmatrix}
σ^3 \wedge σ^2 \\
σ^3 \wedge σ^4 \\
0 \\
0 
\end{pmatrix}.
\]
Our first step: consider just the Lie algebra.
The Lie algebra \(\LieG\) of the structure group \(G\) of a \(G\)-structure is the set of values of the matrix \(\left(γ^i_j\right)\) as we vary its entries.
As \(\LieG\) consists of upper triangular matrices, every Engel plane field on any complex manifold \(M\) determines a filtration of holomorphic vector subbundles
\[
0=\vb{W}_0\subset\vb{W}_1\subset\vb{W}_2\subset\vb{W}_3\subset\vb{W}_4 = TM
\]
forming a complete flag.
Let
\begin{align*}
a &= dγ^3_3, \\
b &= dγ^4_4;
\end{align*}
differential forms which descend to the Chern classes in Dolbeault cohomology of the quotient line bundles \(\vb{W}_3/\vb{W}_2\) and \(\vb{W}_4/\vb{W}_3\).
The Atiyah class of the tangent bundle lies in the same Lie algebra \(\LieG\).
The associated graded of the filtration is represented by the \(\LieG\)-module:
\[
\begin{pmatrix}
2γ^3_3+γ^4_4 & 0 & 0 & 0 \\
0 & γ^3_3+γ^4_4 & 0 & 0 \\ 
0 & 0 & γ^3_3 & 0 \\
0 & 0 & 0 & γ^4_4 
\end{pmatrix}
\]
which is also the Lie algebra \(\LieG_0\) of the obvious Langlands decomposition, i.e. the maximal reductive linear algebraic subgroup.
Taking characteristic polynomial of this matrix, the Chern classes of tangent bundle (or equivalently, of the associated graded of the tangent bundle), in Dolbeault cohomology, are
\begin{align*}
c_1&=\frac{i}{2 \, \pi}(4 \, a + 3 \, b), \\
c_2&=\left( \frac{i}{2 \, \pi}\right)^2(5 \, a^{2} + 9 \, a b + 3 \, b^{2}), \\
c_3&=\left( \frac{i}{2 \, \pi}\right)^3(2 \, a^{3} + 8 \, a^{2} b + 6 \, a b^{2} + b^{3}), \\
c_4&=\left( \frac{i}{2 \, \pi}\right)^4({\left(2 \, a + b\right)} {\left(a + b\right)} a b),
\end{align*}
The reader can check that
\[
0 = 
c_1^4 
-\frac{11}{2}c_1^2c_2
+4c_2^2
+\frac{21}{2}c_1 c_3
-
\frac{75}{2}c_4.
\]
Hence any complex \(4\)-manifold \(M\) which admits an Engel plane field satisfies this equation in the Chern classes of its tangent bundle \(TM\), in Dolbeault cohomology.

Our second step: compute the Lie algebra prolongation \(\prol\LieG\).
We do this without computing the prolongation of these structure equations, so we only arrive at structure equations modulo torsion terms.
Let \(\nabla γ=-ϖ{}\wedge σ\) be the Atiyah form:
\[
\nabla γ=
\begin{pmatrix}
2 \nabla γ^3_3+\nabla γ^4_4 & 0 & 0 & 0 \\
\nabla γ^2_1 & \nabla γ^3_3+\nabla γ^4_4 & 0 & 0 \\ 
\nabla γ^3_1 & \nabla γ^3_2 & \nabla γ^3_3 & 0 \\
\nabla γ^4_1 & \nabla γ^4_2 & \nabla γ^4_3 & \nabla γ^4_4 
\end{pmatrix}.
\]
Compute \(\nabla γ\) components, i.e. compute \(\prol\LieG\):
\[
\nabla
\begin{pmatrix}
γ^2_1\\
γ^3_1\\
γ^3_2\\
γ^3_3\\
γ^4_1\\
γ^4_2\\
γ^4_3\\
γ^4_4
\end{pmatrix}
=
\begin{pmatrix}
ϖ{11}^2 & ϖ{13}^3+ϖ{14}^4 & 0 & 0 \\
ϖ{11}^3 & ϖ{12}^3 & ϖ{13}^3 & 0 \\
ϖ{12}^3 & ϖ{22}^3 & ϖ{23}^3 & 0 \\
ϖ{13}^3 & ϖ{23}^3 & 0 & 0\\
ϖ{11}^4 & ϖ{12}^4 & ϖ{13}^4 & ϖ{14}^4 \\
ϖ{12}^4 & ϖ{22}^4 & ϖ{23}^4 & -2ϖ{23}^3 \\
ϖ{13}^4 & ϖ{23}^4 & ϖ{33}^4 & 0 \\
ϖ{14}^4 & -2ϖ{23}^3 & 0 & 0
\end{pmatrix}
\wedge
\begin{pmatrix}
σ^1\\
σ^2\\
σ^3\\
σ^4
\end{pmatrix}.
\]
Look at the last two diagonal entries to see that in the Atiyah class expression
\begin{align*}
a=dγ^3_3 &= -ϖ{13}^3 \wedge σ^1 - ϖ{23}^3 \wedge σ^2, \\
b=dγ^4_4 &= -ϖ{14}^4 \wedge σ^1 + 2ϖ{23}^3 \wedge σ^2
\end{align*}
(computing modulo torsion and curvature, as they do not affect the Atiyah class).
Each only involves \(σ^1,σ^2\), while \(2a+b\) only involves \(σ^1\).
Therefore \(0=(2a+b)^2\) and any polynomial of degree \(3\) or more in \(a,b\) vanishes.
The associated Chern--Simons form
\begin{align*}
\CS{(2a+b)^2}
&=(2γ^3_3+γ^4_4)\wedge(2\,dγ^3_3+dγ^4_4),\\
&=-(2γ^3_3+γ^4_4)\wedge(2ϖ{13}^3+ϖ{14}^4)\wedge σ^1
\end{align*}
is closed, precisely because of the vanishing of \((2a+b)^2\).
Note also that \(\CS{a}\wedge\CS{b}\) is closed, even though perhaps neither \(\CS{a}\) nor \(\CS{b}\) are.
\begin{theorem}
The tangent bundle \(TM\) of any \(4\)-dimensional complex manifold \(M\) with a holomorphic Engel plane field satisfies  \(0=a(M,T)^3=c_1^3=c_1c_2=c_3=c_2^2=c_4\) in Dolbeault cohomology.
The induced subbundles \(0=\vb{W}_0\subset\vb{W}_1\subset\vb{W}_2\subset\vb{W}_3\subset TM\) have \(c_1(\vb{W}_1)^2=0\).
\end{theorem}
Summing up, the naive calculation using only the identification of the Lie algebra \(\LieG\) gave only one equation, of fourth order, in the Chern classes, while identification of the prolongation \(\prol\LieG\) gives \(5\) equations, \(3\) of third order and \(2\) of fourth order.
\begin{example}
The compact non-K\"ahler \(4\)-fold \(M=\SU{3}\) has \(c_1^3 \ne 0\) in Dolbeault cohomology \cite{Griffiths:1962}, so bears no holomorphic Engel plane field.
\end{example}
\begin{example}
A compact complex manifold with a holomorphic Engel plane field is not of general type.
\end{example}
But we know more, at least in some rough intuitive heuristic form: the Atiyah class of the frame bundle of a complex \(4\)-manifold is computed as a differential form which, in coordinates, has \(64\) components.
But the Atiyah class of an \(G\)-bundle is expressed as a differential form with \(32\) components.
Finally, taking into account the prolongation, there are only \(14\) components.

\section{Example: Baum--Bott}
Suppose that \(\vb{W} \subset TM\) is a holomorphic rank \(p\) subbundle of the tangent bundle of a complex manifold \(M\).
Denote the complex dimension of \(M\) by \(p+q\).
Let \(G \subset \GL{p+q}\) be the subgroup preserving \(W\defeq\C[p] \oplus 0 \subset \C[p+q]\), \(G_0\) the subgroup preserving a complement.
Let \(G \to E \to M\) be the \(G\)-structure consisting of the pairs \((m,u)\) where \(m \in M\) and \(u \colon T_m M \to \C[p+q]\) is a complex linear isomorphism for which \(u(\vb{W}_m) = W\).


If we let
\[
\nabla
\begin{pmatrix}
σ^i\\
σ^I
\end{pmatrix}
\defeq
d
\begin{pmatrix}
σ^i\\
σ^I
\end{pmatrix}
+
\begin{pmatrix}
γ^i_j&γ^i_J\\
0&γ^I_J
\end{pmatrix}
\wedge
\begin{pmatrix}
σ^i\\
σ^I
\end{pmatrix},
\]
then, after absorption of torsion, 
\[
\nabla
\begin{pmatrix}
σ^i\\
σ^I
\end{pmatrix}
=
\frac{1}{2}
\begin{pmatrix}
0\\
t^I_{jk}σ^j\wedge σ^k
\end{pmatrix},
\]
so the torsion is \(t^I_{jk}\).
We leave the reader to check that the torsion is anchored if and only if it vanishes, which occurs if and only if the subbundle \(\vb{W}\subset TM\) is bracket closed, i.e. a holomorphic foliation, and the only possible anchor is \(t=0\).
We then find that, if we set
\begin{align*}
\nabla
\begin{pmatrix}
γ^i_j&γ^i_J\\
0&γ^I_J
\end{pmatrix}
&\defeq
d
\begin{pmatrix}
γ^i_j&γ^i_J\\
0&γ^I_J
\end{pmatrix}
+
\begin{pmatrix}
γ^i_k&γ^i_K\\
0&γ^I_K
\end{pmatrix}
\wedge
\begin{pmatrix}
γ^k_j&γ^k_J\\
0&γ^K_J
\end{pmatrix}
\\
&\qquad
+
\begin{pmatrix}
ϖ{jk}^i\wedge σ^k+ϖ{jK}^i\wedge σ^k&
ϖ{Jk}^i\wedge σ^k+ϖ{JK}^i\wedge σ^k\\
0&
ϖ{Jk}^I\wedge σ^k+ϖ{JK}^I\wedge σ^k\\
\end{pmatrix}
\end{align*}
then
\[
\nabla
\begin{pmatrix}
γ^i_j&γ^i_J\\
0&γ^I_J
\end{pmatrix}
=0
\]
for any local choice of prolongation pseudoconnection, which is not surprising, as all holomorphic foliations are locally isomorphic.
The \(1\)-forms \(γ^I_J\) are the holomorphic connection on the normal bundle of each leaf.

Take any \(\GL{q}\)-invariant polynomial \(P\) of degree \(\ge q+1\), perhaps valued in a finite dimensional holomorphic \(\GL{q}\)-module.
Write, as above,
\[
\nabla γ^I_J = dγ^I_J + γ^I_K \wedge γ^K_J.
\]
We find
\[
P(\nabla γ^I_J)=P(-γ^I_{JK} \wedge γ^K),
\]
expands out to have more than \(q\) \(1\)-forms \(\omega^K\) in each term.
But there are only \(q\) such \(1\)-forms, so 
\[
P(\nabla γ^I_J)=0
\]
modulo torsion.
The Chern--Simons form is then
\[
\CS{P}=P(γ^I_J,\nabla γ^I_J,\dots,\nabla γ^I_J)
\]
which also vanishes, modulo torsion, if there are more than \(q\) \(1\)-forms \(\omega^K\) in each term, i.e. if \(P\) has degree \(q+2\) or more.
We recover the Baum--Bott theorem \cite{Baum/Bott:1972} p. 287 for holomorphic folations, with results of Kamber and Tondeur \cite{Kamber/Tondeur:1975}:
\begin{theorem}
All Chern classes, in Dolbeault cohomology, of the normal bundle of any holomorphic foliation, of degree more than the codimension of the foliation, vanish.
All of their associated Chern--Simons classes, in Dolbeault cohomology, 
of degree at least two more than the codimension of the foliation, vanish.
\end{theorem}
Again, we stress that this theorem is a direct consequence of the linear algebra computation of \(\prol\LieG\) for \(G \subset \GL{p+q}\) the stabilizer of a \(p\)-dimensional linear subspace.

\section{Example: Baum--Bott with volume form}
We want to see our theory give rise to new results similar to the Baum--Bott vanishing theorem.
If \(F\) is a holomorphic foliation on a complex manifold \(M\) equipped with a holomorphic volume form, we can write the foliation and volume form together as a holomorphic \(G\)-structure where \(G \subset \SL{p+q}\) is the group of unimodular complex linear transformations preserving \(\C[p]+0\subset\C[p+q]\).
Calculate \(\prol\LieG\) to see that, in the standard flat geometry, the expression \(a=-\omega^i_{jk} \wedge σ^k\) becomes
\[
-
\begin{pmatrix}
ϖ{jk}^i \wedge σ^k + ϖ{jK}^i \wedge σ^K & 
ϖ{Jk}^i \wedge σ^k + ϖ{JK}^i \wedge σ^K 
\\
0 &
\pr{ϖ{JK}^I-\frac{1}{q+1}\pr{\delta^I_J ϖ{iK}^i+\delta^I_K ϖ{iJ}^i}} \wedge σ^K 
\end{pmatrix}
\]
where \(ϖ{ij}^i=0\) and \(ϖ{jk}^i=ϖ{kj}^i\) and \(ϖ{IJ}^I=0\)  and \(ϖ{JK}^I=ϖ{KJ}^I\).
Therefore if we write
\[
c_1(TF) = ϖ{iK}^i\wedge σ^K
\]
then the \((1,1)\)-part of this form descends to Dolbeault cohomology to represent \(c_1(TF)\).
Clearly as above
\begin{align*}
c_1(TF)^{q+1}=0,\\
T_{c_1(TF)}^{q+2}=0
\end{align*}
in addition to the results we saw previously from the Baum--Bott theorem.

\section{Example: scalar conservation laws}
Bryant, Griffiths and Hsu \cite{Bryant/Griffiths/Hsu:1995} constructed out of any scalar conservation law an equivalent \(G\)-structure. 
Their \(G\)-structure has structure equations (in a slight alteration of their notation)
\[
 d
\begin{pmatrix}
σ^1 \\
σ^2 \\
σ^3
\end{pmatrix}
= 
-\begin{pmatrix}
2 ω1 & 0 & 0 \\
0 & ω1 & 0 \\
0 & ω2 & - ω1
\end{pmatrix}
\wedge 
\begin{pmatrix}
σ^1 \\
σ^2 \\
σ^3
\end{pmatrix}
+
\begin{pmatrix}
K σ^2 \wedge σ^3 \\
σ^1 \wedge σ^3  \\
0
\end{pmatrix}.
\]
They consider a real scalar conservation law.
We will consider a holomorphic scalar conservation law, for which exactly the same derivation yields a holomorphic \(G\)-structure.
In our notation,
\begin{align*}
σ&=
\begin{pmatrix}
σ^1 \\
σ^2 \\
σ^3
\end{pmatrix},
\\
γ&=
\begin{pmatrix}
2 γ_1 & 0 & 0 \\
0 & γ_1 & 0 \\
0 & γ_2 & - γ_1
\end{pmatrix},
\\
γ_0&=
\begin{pmatrix}
2 γ_1 & 0 & 0 \\
0 & γ_1 & 0 \\
0 & 0 & - γ_1
\end{pmatrix},
\end{align*}
where \(γ_0\) is the projection to the Lie algebra of the maximal reductive subgroup.
Take the prolongation of the Lie algebra to get
\[
dγ+γ\wedge γ
=
-\begin{pmatrix}
0 & 0 & 0 \\
0 & 0 & 0 \\
0 & ϖ{32}^3 \wedge σ^2 & 0
\end{pmatrix}
=-ϖ{}\wedge σ,
\]
modulo torsion.
Clearly all Chern classes of the tangent bundle vanish, so a complex \(3\)-manifold which is the phase space of a holomorphic conservation law has all Chern classes of its tangent bundle vanish in Dolbeault cohomology.
It is not clear whether the Atiyah class of the tangent bundle vanishes, but \(a(M,T)^2=0\).
There is a trivial characteristic class \(dω1=0\) in Dolbeault cohomology, so that the various invariant subbundles of the tangent bundle, forming a flag, have holomorphic connections on their associated graded line bundles.
There is a possibly nontrivial characteristic class \(dω2=-\omega^3_{32}\wedge σ^2\), vanishing on the leaves of \(0=σ^2=σ^3\), so on those leaves, the tangent bundle of the \(3\)-fold \(M\) pulls back to split into a direct sum, with a holomorphic affine connection.

Clearly the Chern--Simons classes in Dolbeault cohomology \(\CS{c_1},\CS{c_2},\CS{c_3}\) of the tangent bundle all vanish as well.
For example,
\begin{align*}
(-2\pi i)^2\CS{c_2}
&=
\omega^i_j\wedge\nabla\omega^j_i,
\\
&=
-
\operatorname{tr}
\begin{pmatrix}
2 γ1 & 0 & 0 \\
0 & γ1 & 0 \\
0 & γ2 & -γ1
\end{pmatrix}
\wedge
\begin{pmatrix}
0 & 0 & 0 \\
0 & 0 & 0 \\
0 & ϖ{32}^3 \wedge σ^2 & 0
\end{pmatrix},
\\
&=
0.
\end{align*}

\section{Example: projective Baum--Bott}
We modify the discussion above to consider a holomorphic foliation with transverse normal projective connection.
Take a rank \(p\) holomorphic foliation on a complex manifold \(M\) of dimension \(p+q\), with a transverse holomorphic projective connection.
From the Baum--Bott theory, every polynomial in Chern classes of degree (as a differential form) exceeding \(q+1\) vanishes in Dolbeault cohomology.
We leave the reader to justify the structure equations:
\[
\nabla
\begin{pmatrix}
σ^i \\
σ^I
\end{pmatrix}
=
0
\]
and
\[
\nabla
\begin{pmatrix}
γ^i_j & γ^i_J \\
0 & γ^I_J
\end{pmatrix}
=
-\begin{pmatrix}
ϖ{jk}^i \wedge σ^k+ϖ{jK}^i \wedge σ^K & 
ϖ{Jk}^i \wedge σ^k+ϖ{JK}^i \wedge σ^K \\
0 & 
\pr{\delta^I_J ϖK + \delta^I_K ϖJ} \wedge σ^K
\end{pmatrix}
\]
modulo torsion.

The normal bundle of the foliation is the associated vector bundle associated the the representation
\[
\rho
\begin{pmatrix}
γ^i_j & γ^i_J \\
0 & γ^I_J
\end{pmatrix}
= γ^I_J.
\]
The Atiyah class of the normal bundle is therefore represented in Dolbeault cohomology by
\begin{align*}
\nabla γ^I_J 
&= 
dγ^I_J + γ^I_K \wedge γ^K_J,
\\
&= - \delta^I_J ϖK \wedge σ^K - ϖJ \wedge σ^I
\end{align*}
modulo torsion.

The example of \(M=\C[p] \times \Proj{q}\) with the obvious foliation and tranverse projective connection obviously has translations and projective transformations acting transitively on the total space of every prolongation.
Looking at the example, the Atiyah class of the normal bundle is just computed precisely as the Atiyah class of the tangent bundle of \(\Proj{q}\), since it only involves the \(1\)-forms with capital letter superscripts and subscripts.
In particular, besides the results from the Baum--Bott theorem, we find:
\begin{theorem}
Take a holomorphic foliation \(F\) of rank \(p\) on a complex manifold \(M\) of complex dimension \(p+q\), with transverse holomorphic projective connection.
If \(N\) is the normal bundle, then, in Dolbeault cohomology:
\[
i(q+1)a(N) = 2\pi I \otimes c_1(N) + 2\pi c_1(N) \otimes I
\]
and
\[
\binom{q+1}{k}c_1^k(N)
=
(q+1)^k c_k(N).
\]
\end{theorem}
These equations hold for the tangent bundle of \(\Proj{q}\), and so we correctly predict them here; plugging in the structure equations, they pop out.
\section{Example: split tangent bundle}
\begin{theorem}
Suppose that \(G\subset\GL{n}\) is a reductive linear algebraic group.
Take any polynomial which vanishes on the infinitesmal characteristic forms of \(G\) as defined in \vref{subsec:inf.char}.
Then that polynomial vanishes on the Chern classes in Dolbeault cohomology of any complex manifold which admits a holomorphic \(G\)-structure.
\end{theorem}
\begin{proof}
Every holomorphic \(G\)-structure is anchored, by splitting \(\LieG \otimes \C[n*]\) into irreducible \(G\)-modules.
We fix one such anchoring of all \(G\)-structures and apply lemma~\vref{lemma:main}.
\end{proof}
We recover \cite{Beauville:2000} p. 3, Lemma 3.1; our proof is longer, but only because we wish to be very explicit in computing the structure equations.
\begin{theorem}
Suppose that \(M\) is a complex manifold and that \(V\subset TM\) is a holomorphic direct summand of \(TM\).
Then the Atiyah class \(a_V\) of \(V\) lies in \(\cohomology{1}{M,V^*\otimes V^*\otimes V}\).
In particular, every class in \(\cohomology{r}{M,\Omega^r}\) given by a polynomial in the Chern classes of \(V\) in Dolbeault cohomology vanishes for \(r\) exceeding the rank of \(V\). 
\end{theorem}
\begin{proof}
Suppose that \(TM=V\oplus W\).
Let \(E\) be the set of pairs \((m,u)\) for \(m\in M\) and \(u\colon T_m M \to \C[n]\) a linear isomorphism taking \(V_m,W_m\) to some fixed complementary linear subspaces \(\C[p]\oplus 0, 0\oplus\C[q]\subset\C[n]\), \(n=p+q\), so \(E\) is a holomorphic \(G\)-structure where \(G=\GL{p}\times\GL{q}\subset\GL{n}\) is the set of linear transformations preserving those subspaces.
We prove the stronger result that the Atiyah form of the associated holomorphic \(G\)-structure lies in a direct sum.
(Roughly speaking, the computation of the Chern classes in Dolbeault cohomology proceeds, as in our general theory above, exactly as if there were no torsion.
For a splitting of the tangent bundle, this means as if the splitting \(V\oplus W\) were bracket closed, i.e. locally a product.)

Compute the prolongation:
\[
\prol\LieG=\prol{\LieGL{p}}\oplus\prol{\LieGL{q}}.
\]
The proof is then just to compute the pairing \(\prol\LieG \otimes \C[n]\to\LieG\) to see the Chern form equations.
To be more explicit, we work out the complete structure equations.
In indices
\begin{align*}
i,j,k,\ell&=1,2,\dots,p,\\
I,J,K,L&=p+1,p+2,\dots,p+q=n,
\end{align*}
write the structure equations:
\[
d
\begin{pmatrix}
\sigma^i\\
\sigma^I
\end{pmatrix}
+
\begin{pmatrix}
γ^i_j&0\\
0&γ^I_J
\end{pmatrix}
\wedge
\begin{pmatrix}
\sigma^j\\
\sigma^J
\end{pmatrix}
=
\frac{1}{2}
\begin{pmatrix}
t^i_{JK}\sigma^J\wedge\sigma^K\\
t^J_{jk}\sigma^j\wedge\sigma^k
\end{pmatrix},
\]
with anchor being the vanishing of all other torsion components, i.e. we have absorbed torsion.
Differentiating the structure equations yields relations on the \(1\)-torsion and \(2\)-torsion:
\[
0=-(dγ^i_j+γ^i_k\wedge γ^k_j+ϖ{jk}^i\wedge\sigma^k)\wedge\sigma^j
+
\frac{1}{2}
(
dt^i_{JK}
-t^i_{LK}γ^L_J
+t^i_{JL}γ^L_K
)\wedge\sigma^J\wedge\sigma^K.
\]
The Atiyah forms of \(V\) and of \(W=TM/V\) are:
\begin{align*}
a_V&=dγ^i_j+γ^i_k\wedge γ^k_j=-ϖ{jk}^i\wedge\sigma^k,\\
a_W&=dγ^I_J+γ^I_K\wedge γ^K_J=-ϖ{JK}^I\wedge\sigma^K,\\
\end{align*}
modulo \(1\)-torsion and \(2\)-torsion.
There are no capital letter indices in \(a_V\), which is the crucial observation.
The Atiyah class \(a_V\) in Dolbeault cohomology is represented by differences of these \(-(ϖ{jk}^i)^{0,1}\wedge\sigma^k\) across local holomophic sections of \(E\).
The \(\sigma^k\) on each pulls back to a local holomorphic section of \(V^*\), while \((ϖ{jk}^i)^{0,1}\) pulls back to a local smooth \((0,1)\) section of \(V^*\otimes V\).
If we wedge more than \(p\) of these together, we wedge together more than \(p\) \(\sigma^i\) \(1\)-forms, but there are only \(p\) of these.
\end{proof}
\section{Conclusion}
The reader can construct a purely holomorphic theory of \Slovak{} cohomology for \(G\)-structures, by imitation of the theory for Cartan geometries \cite{mckay2022}.

It remains to define \(G\)-structures on singular varieties, generalizing the theory of singular locally Hermitian symmetric varieties, on which some results about characteristic class invariants are known \cite{Mumford:1977} which generalize Hirzebruch's proportionality theorem.
\nocite{sagemath}
\bibliographystyle{amsplain}
\bibliography{chern-cartan-3}
\end{document}